\documentclass[12pt]{article}
\usepackage[margin=35mm]{geometry}
\usepackage{amsmath}
\usepackage{amsfonts}
\usepackage{amssymb}
\usepackage{amsthm}
\usepackage{graphicx}
\usepackage{verbatim}
\usepackage{tikz}
\immediate\write18{texcount -tex -sum  \jobname.tex > \jobname.wordcount.tex}

\newtheorem{theorem}{Theorem}
\newtheorem{lemma}{Lemma}

\newtheorem{corollary}{Corollary}

\newtheorem{question}{Question}

\providecommand{\keywords}[1]
{
  \small	
  \textbf{\textit{Keywords---}} #1
}

\title{On the minimum number of maximal distance-$k$ independent sets in trees}
\author{Dmitrii Taletskii$^{1}$\\
\small
$^1$ National Research University Higher School of Economics, \\
\small
Bolshaya Pechyorskaya ul. 25/12, Nizhny Novgorod, 603155 Russia\\%
    \small
    e-mail: dmitalmail@gmail.com
}
\date{} 

\begin{document}
\maketitle

\begin{abstract}
A vertex subset of a graph is called a \textit{distance-$k$ independent set} if the distance between any two of its distinct vertices is at least $k + 1$. For all $n,k \geq 1$, we determine the minimum possible number of inclusion-wise maximal distance-$k$ independent sets among all $n$-vertex trees. It equals~$n$ if $n \leq k + 1$, and $n - \bigg\lfloor \frac{n - (k \bmod 2)}{\lfloor k/2 \rfloor + 1} \bigg\rfloor + 1$ otherwise. We also completely describe the class of trees attaining this bound and determine the growth rate of the number of such $n$-vertex trees for a fixed $k \geq 1$. If $k$ is odd and $(k+1)/2$ does not divide $n-1$, then the number of non-isomorphic $n$-vertex trees with the minimum possible number of maximal distance-$k$ independent sets grows linearly with~$n$. Otherwise, it is bounded above by the number of unlabeled \mbox{$k^2$-vertex} trees.
\end{abstract} 


\section{Introduction}

Let $G$ be a simple graph with vertex set $V(G)$ and edge set $E(G)$. A set $I \subseteq V(G)$ is \textit{independent} if no two vertices in $I$ are adjacent. An independent set is a \textit{maximal independent set} (MIS) if it is inclusion-wise maximal. A set $J \subseteq V(G)$ is a \textit{distance-$k$ independent set} ($k$-DIS) if the distance between any two distinct vertices in $J$ is at least $k + 1$. An inclusion-wise maximal $k$-DIS is called a maximal distance-$k$ independent set ($k$-MDIS). Note that a (maximal) distance-$1$  independent set is just a (maximal) independent set.
 
The maximum possible number of MISs in the class of $n$-vertex graphs was obtained by Miller and Muller~\cite{MM60} and, independently, by Moon and Moser~\cite{MM65}. Since then, maximum values have also been determined for various graph classes, such as connected graphs, (connected) triangle-free graphs, (connected) unicyclic graphs, trees and forests~\cite{F87}--\cite{CJ97}. While the sharp upper bounds typically grow exponentially with the number of vertices, the sharp lower bounds are constant for many graph classes. This is because \textit{twin vertices} (that is, two vertices with the same open neighborhood) do not increase the total number of MISs. For example, an $n$-vertex star graph has exactly two MISs for all $n \geq 2$. However, the sharp lower bounds are nontrivial for various classes of twin-free graphs. In~\cite{TM18}, an exponential lower bound for the class of $n$-vertex twin-free trees was established. In~\cite{CW24}, a logarithmic lower bound for the class of all twin-free graphs and a linear lower bound for the class of bipartite twin-free graphs were obtained.

Every $n$-vertex tree has at least $n + 1$ $k$-DISs, and this bound is sharp for all trees of diameter at most $k$. In~\cite{T23}, for all $1 < k < d < n$, the trees $T_{k,d,n}$ with the minimum possible number of $k$-DISs among all $n$-vertex trees of diameter~$d$ were described. These trees are constructed from the path $P_{d+1}$ by connecting $n - d - 1$ leaves to a $k$-th vertex if $d > 2k - 2$, or to a central vertex otherwise. In contrast, the trees with the minimum possible number of $1$-DISs among all $n$-vertex trees of diameter~$d$ were described only for $d \leq 7$, and their structure is much more complex~\cite{FPSV13, T21}.

To date, there are relatively few results on enumerating $k$-MDISs in various graph classes. In~\cite{E13}, the number of $k$-MDISs was calculated for the class of grid graphs. In~\cite{BBMB24}, an efficient distributed algorithm was proposed to enumerate $2$-MDISs. Also, there are several complexity results for the problem of finding a $k$-MDIS with maximum cardinality~\cite{EGM14}. 

For $n \leq k + 1$, every $n$-vertex tree has exactly $n$ $k$-MDISs. In this paper, we show that, if $1 \leq k \leq n - 2$, then the sharp lower bound on the number of $k$-MDISs for the class of $n$-vertex trees equals

$$n - \bigg\lfloor \frac{n - (k \bmod 2)}{\lfloor k/2 \rfloor + 1} \bigg\rfloor + 1.$$

For all $k \geq 1$, we also provide a complete characterization of the trees that attain this lower bound (which we refer to as \textit{$k$-minimal trees}). It turns out that if $k$ is odd and $(k+1)/2$ does not divide $n - 1$, then the number of $k$-minimal $n$-vertex trees is~$\Theta(n)$. Otherwise, it is bounded by a constant that depends only on~$k$. Furthermore, for $n \geq 4$, there is a unique $k$-minimal $n$-vertex tree if and only if $n \geq k + 2$ and $\lfloor k/2 \rfloor + 1$ divides $n - (k \bmod 2)$. As a simple corollary, we show that for all $k \geq 2$ and $n \geq 5$, the sharp lower bound on the number of $k$-MDISs in the class of isolate-free $n$-vertex forests is identical to the bound for trees. We also give a few remarks on the lower bounds for arbitrary and bipartite isolate-free graphs.

\section{Terminology}

Let $G$ be a graph and let $v \in V(G)$. For every $s \geq 1$, let $N_s[v]$ and $N_s(v)$ denote the sets of vertices at distance at most $s$ and exactly $s$ from $v$, respectively. A vertex $u$ is \textit{$k$-covered} by a vertex subset $J$ if there exists a vertex $v \in J$ such that $u \in N_k[v]$. If the parameter $k$ is clear from context, we use the term \textit{covered}. We write $G_1 \subseteq G_2$ if $G_1$ is a subgraph of $G_2$, and $G_1 \nsubseteq G_2$ otherwise. Let $[n] = \{1,\dots,n\}$.

A vertex of a tree is a \textit{leaf} if it has degree 1 and is a \textit{branching vertex} if it has degree at least 3. A leaf is  \textit{diametral} or \textit{central} if it has maximum or minimum eccentricity among all leaves, respectively. An inclusion-wise maximal subtree of a tree that does not contain its central vertices is called a \textit{main subtree}. An inclusion-wise maximal path of a tree that contains a leaf but no branching vertices is called a \textit{pendant path}.
 
Let $S_n$ and $P_n$ denote an $n$-vertex star and a simple $n$-vertex path (or \textit{$n$-path}), respectively. Let $S_{n,m}$ be an $(nm+1)$-vertex tree obtained from the forest $nP_m$ by adding a new vertex and connecting it to one leaf of each $P_m$. Note that $S_{n,1} \cong S_{n+1}$ and $S_{2,m} \cong P_{2m + 1}$. Let $S'_{n,m}$ denote the $(mn)$-vertex subtree of $S_{n,m}$ obtained by removing an arbitrary leaf. Let $B_{p_1,p_2,s}$ be the tree of diameter $2s + 1$ with $(p_1 + p_2) \cdot s + 2$ vertices, such that its two central vertices are connected to $p_1$ and $p_2$ simple $s$-paths, respectively, where $p_1,p_2 \geq 1$. Let $\mathcal{B}_{p,s}$ be the family of all trees $B_{p_1,p_2,s}$ such that $p_1 + p_2 = p$ (see the family $\mathcal{B}_{5,2}$ in Fig.~1).

\begin{figure}[h!]
\center
\begin{tikzpicture}
  [scale=1,auto=left,every node/.style={circle}]

\node[fill = black,inner sep=2.5pt] (a1) at (0,0.5) {};
\node[fill = black,inner sep=2.5pt] (a2) at (1,0.5) {};
\node[fill = black,inner sep=2.5pt] (a3) at (0,1.5) {};
\node[fill = black,inner sep=2.5pt] (a4) at (1,1.5) {};
\node[fill = black,inner sep=2.5pt] (a5) at (0,-0.5) {};
\node[fill = black,inner sep=2.5pt] (a6) at (1,-0.5) {};
\node[fill = black,inner sep=2.5pt] (a7) at (0,-1.5) {};
\node[fill = black,inner sep=2.5pt] (a8) at (1,-1.5) {};
\node[fill = black,inner sep=2.5pt] (b1) at (2,0) {};
\node[fill = black,inner sep=2.5pt] (b2) at (3,0) {};
\node[fill = black,inner sep=2.5pt] (c1) at (4,0) {};
\node[fill = black,inner sep=2.5pt] (c2) at (5,0) {};
\node[fill = black,inner sep=2.5pt] (qa1) at (8,0) {};
\node[fill = black,inner sep=2.5pt] (qa2) at (9,0) {};
\node[fill = black,inner sep=2.5pt] (qa3) at (8,1) {};
\node[fill = black,inner sep=2.5pt] (qa4) at (9,1) {};
\node[fill = black,inner sep=2.5pt] (qa5) at (8,-1) {};
\node[fill = black,inner sep=2.5pt] (qa6) at (9,-1) {};
\node[fill = black,inner sep=2.5pt] (qb1) at (10,0) {};
\node[fill = black,inner sep=2.5pt] (qb2) at (11,0) {};
\node[fill = black,inner sep=2.5pt] (qc1) at (12,-0.75) {};
\node[fill = black,inner sep=2.5pt] (qc2) at (13,-0.75) {};
\node[fill = black,inner sep=2.5pt] (qc3) at (12,0.75) {};
\node[fill = black,inner sep=2.5pt] (qc4) at (13,0.75) {};

\node[fill = none,inner sep=2.5pt] (cap1) at (10.5,-2.5) {$B_{3,2,2}$};

\node[fill = none,inner sep=1.5pt] (cap2) at (2.5,-2.5) {$B_{4,1,2}$};

\foreach \from/\to in {a1/a2,a3/a4,a5/a6,a7/a8,a8/b1,a2/b1,a4/b1,a6/b1,b1/b2,b2/c1,c1/c2,qa1/qa2,qa3/qa4,qa5/qa6,qa2/qb1,qa4/qb1,qa6/qb1,qb1/qb2,qb2/qc1,qb2/qc3,qc1/qc2,qc3/qc4}
    \draw[thick] (\from) -- (\to);

\end{tikzpicture}
\caption{\small  The family $\mathcal{B}_{5,2}$.}
\end{figure}
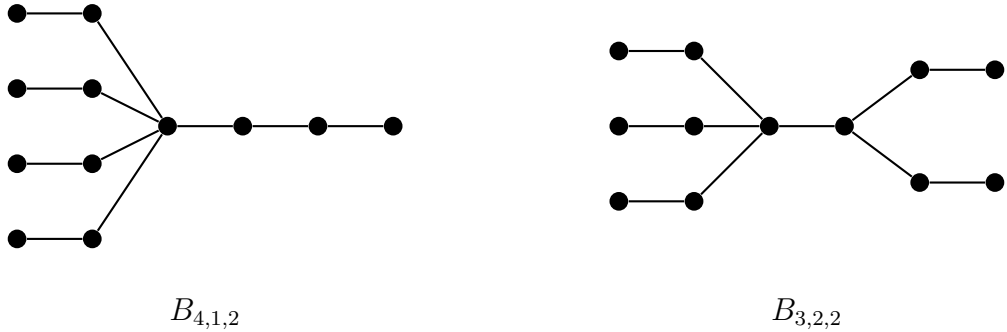

Let $T$ be a tree and let $k \geq 2$. A vertex $u \in V(T)$ is a \textit{$k$-twin} if there exists a vertex $v \in V(T) \setminus \{u\}$ such that $N_k[u] = N_k[v]$. A pair of vertices $(x,y)$ is \textit{$k$-special} if $x$ and $y$ are diametral leaves that are both $k$-twins, but $N_k[x] \neq N_k[y]$. A vertex is called a \textit{diametral $k$-twin} if it is a diametral leaf and a $k$-twin.

Let $T$ be an $n$-vertex tree with no diametral $k$-twins for some $k \geq 2$, and let $r \geq 0$. Let $\mathrm{Add}_k(T,r)$ denote the family of all $(n + r)$-vertex trees of diameter $\mathrm{diam}(T)$ that contain $T$ as a subgraph and do not contain diametral $k$-twins. Similarly, let $\mathrm{Add}^*_k(T,r)$ denote the family of all $(n + r)$-vertex trees of diameter $\mathrm{diam}(T)$ that contain $T$ as a subgraph and do not contain $k$-special pairs. In other words, $\mathrm{Add}_k(T,r)$ ($\mathrm{Add}^*_k(T,r)$) is the family of trees obtained by adding $r$ new vertices to $T$ in a way that preserves the original diameter and ensures the resulting tree does not contain diametral $k$-twins ($k$-special pairs). Note that $\mathrm{Add}_k(T,r) \subseteq \mathrm{Add}^*_k(T,r)$ by definition. For a class of trees $\mathcal{T}$, let $\mathrm{Add}_k(\mathcal{T},r) = \bigcup_{T \in \mathcal{T}}\mathrm{Add}_k(T,r)$. The family $\mathrm{Add}^*_2(S'_{3,2},1) = \{T_1, T_2, T_3\}$ is depicted in Fig.~2. Observe that $T_3$ contains diametral $2$-twins; hence, it does not belong to the family $\mathrm{Add}_2(S'_{3,2},1) = \{T_1, T_2\}$.

\begin{figure}[h!]
\center
\begin{tikzpicture}
  [scale=1,auto=left,every node/.style={circle}]

\node[fill = black,inner sep=2.5pt] (a0) at (1,1) {};
\node[fill = black,inner sep=2.5pt] (a1) at (0,0) {};
\node[fill = black,inner sep=2.5pt] (a2) at (0,-1) {};
\node[fill = black,inner sep=2.5pt] (a3) at (1,0) {};
\node[fill = black,inner sep=2.5pt] (a4) at (1,-1) {};
\node[fill = black,inner sep=2.5pt] (a5) at (2,0) {};
\node[fill = gray,inner sep=1.75pt] (a6) at (2,-1) {};

\node[fill = black,inner sep=2.5pt] (b0) at (5.5,1) {};
\node[fill = black,inner sep=2.5pt] (b1) at (4.5,0) {};
\node[fill = black,inner sep=2.5pt] (b2) at (4.5,-1) {};
\node[fill = black,inner sep=2.5pt] (b3) at (5.5,0) {};
\node[fill = black,inner sep=2.5pt] (b4) at (5.5,-1) {};
\node[fill = gray,inner sep=1.75pt] (b5) at (6.5,1) {};
\node[fill = black,inner sep=2.5pt] (b6) at (6.5,0) {};

\node[fill = black,inner sep=2.5pt] (c0) at (10,1) {};
\node[fill = black,inner sep=2.5pt] (c1) at (9,0) {};
\node[fill = black,inner sep=2.5pt] (c2) at (9,-1) {};
\node[fill = black,inner sep=2.5pt] (c3) at (10,0) {};
\node[fill = black,inner sep=2.5pt] (c4) at (10,-1) {};
\node[fill = gray,inner sep=1.75pt] (c5) at (11,-1) {};
\node[fill = black,inner sep=2.5pt] (c6) at (11,0) {};

\node[fill = none, inner sep=0.5pt] (cap3) at (1,-2) {\small{$T_1$}};
\node[fill = none, inner sep=0.5pt] (cap3) at (5.5,-2) {\small{$T_2$}};
\node[fill = none, inner sep=0.5pt] (cap3) at (10,-2) {\small{$T_3$}};

\foreach \from/\to in {
a0/a1,a1/a2,a0/a3,a3/a4,a0/a5,a5/a6, b0/b1,b1/b2,b0/b3,b3/b4,b0/b5,b0/b6,c0/c1,c1/c2,c0/c3,c3/c4,c3/c5,c0/c6}
    \draw[thick] (\from) -- (\to);

\end{tikzpicture}
\caption{\small  The family $\mathrm{Add}^*_2(S'_{3,2},1)$.}

\end{figure}
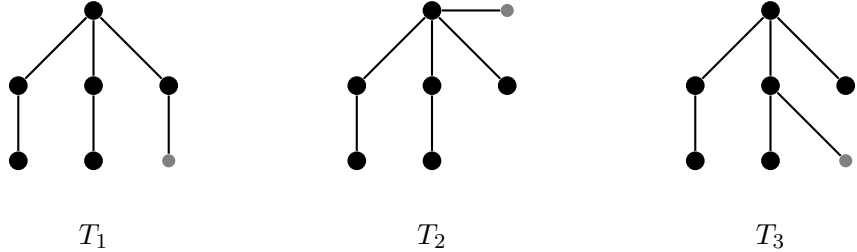

Let $T$ be a tree with a leaf $\ell$. Let $\mathrm{mdi}_k(T)$ denote the total number of $k$-MDISs of~$T$. Let $\mathrm{mdi}_k^*(T,\ell)$ denote the number of $k$-MDISs $J$ of $T$ with the following property: $\ell \in J$ and there exists a vertex $w \in V(T) \setminus J$ such that $N_k[w] \cap J = \{\ell\}$. Let $\mathcal{M}_k(T)$ denote the family of all $k$-MDISs of $T$.

Let $k \geq 2$ and $m \geq 1$. If a tree $T'$ can be obtained from a tree $T$ by deleting a $k$-twin leaf of $T$, we write $T \succ_k T'$. A tree $T$ is called \textit{$(k,m)$-reducible} if there exist trees $T_1,\dots, T_m$ such that $T \succ_k T_1 \succ_k \cdots \succ_k T_m$. If the parameter $k$ is clear from context, we use the terms \textit{$m$-reducible} and \textit{reducible} instead of \textit{$(k,m)$-reducible} and \textit{$(k,1)$-reducible}, respectively. 

Let $t(m)$ denote the number of unlabeled $m$-vertex trees. 

\section{Preliminaries}

For a fixed $k \geq 1$, define the function $f_k(n)$ as follows: if $n \leq k + 1$, let $f_k(n) = n$; if $n \geq k + 2$, let $$f_k(n) = n - \bigg\lfloor \frac{n - (k \bmod 2)}{\lfloor k/2 \rfloor + 1} \bigg\rfloor + 1.$$ Our goal is to prove that $f_k(n)$ is the minimum possible number of $k$-MDISs among all $n$-vertex trees. This is trivially true for $k = 1$; in what follows, we assume that $k \geq 2$. The next simple observation is used frequently later.

\begin{lemma}\label{lem1}
Let $k \geq 2$ and $n \geq k + 2$. The following hold:

(i) The equality $f_k(n) = f_k(n-1)$ holds if and only if $\lfloor k/2 \rfloor + 1$ divides $n - (k \bmod 2)$. Otherwise, $f_k(n) = f_k(n-1) + 1$;

(ii) $f_k(n + \lfloor k/2 \rfloor + 1) = f_k(n) + \lfloor k/2 \rfloor$.
\end{lemma}

\begin{lemma}\label{lem2}
For all $k,p \geq 2$, the following hold:

(i) $\mathrm{mdi}_k(S'_{p,\lfloor k/2 \rfloor+1}) = f_k(p \cdot (\lfloor k/2 \rfloor+1))$; 

(ii) $\mathrm{mdi}_k(S_{p,\lfloor k/2 \rfloor+1}) = f_k(p \cdot (\lfloor k/2 \rfloor+1) + 1)$;

(iii) If $T \in \mathcal{B}_{p,\lfloor k/2 \rfloor+1}$, then  $\mathrm{mdi}_k(T) \geq f_k(p \cdot (\lfloor k/2 \rfloor + 1) + 2)$, with equality if and only if~$k$ is odd.

\end{lemma}

\begin{proof}
We prove only (iii); the other statements follow similarly. Consider a tree $B_{p_1,p_2,\lfloor k/2 \rfloor+1} \in \mathcal{B}_{p,\lfloor k/2 \rfloor+1}$ and let $n = |V(B_{p_1,p_2,\lfloor k/2 \rfloor+1})| = p \cdot (\lfloor k/2 \rfloor+1) + 2$. 

If $k$ is odd, then it is easy to check that for every non-leaf vertex of $B_{p_1,p_2,\lfloor k/2 \rfloor+1}$, there exists a unique $k$-MDIS containing this vertex (and possibly some leaves). Moreover, there exists a unique $k$-MDIS consisting of all leaves of the tree. Therefore, 
$$\mathrm{mdi}_k(B_{p_1,p_2,\lfloor k/2 \rfloor+1}) = n - (p_1 + p_2) + 1 = n - \frac{n - 2}{\lfloor k/2 \rfloor + 1} + 1 = n - \bigg\lfloor \frac{n - 1}{\lfloor k/2 \rfloor + 1} \bigg\rfloor + 1  = f_k(n).$$

If $k$ is even, there exists a unique $k$-MDIS containing all leaves of the tree. Moreover, for every non-leaf vertex, there exists a $k$-MDIS containing this vertex and no other non-leaf vertex. Finally, there exist $p_1 \cdot p_2$ distinct $k$-MDISs containing two non-leaf vertices adjacent to leaves. Therefore, $$\mathrm{mdi}_k(B_{p_1,p_2,\lfloor k/2 \rfloor+1}) >  n - (p_1 + p_2) + 1 = n - \frac{n - 2}{\lfloor k/2 \rfloor + 1}  + 1 \geq n - \bigg\lfloor \frac{n}{\lfloor k/2 \rfloor + 1} \bigg\rfloor + 1 = f_k(n).$$
\end{proof}

\begin{lemma}\label{lem3}
Let $T$ be a tree with a leaf $\ell$. Then, for all $k \geq 1$, the following holds: $$\mathrm{mdi}_k(T) = \mathrm{mdi}_k(T \setminus \ell) + \mathrm{mdi}^*_k(T,\ell).$$
\end{lemma}

\begin{proof}

Let $\mathcal{M}_k(T) = \mathcal{M}^1_k(T) \cup \mathcal{M}^2_k(T)$, where $\mathcal{M}^1_k(T)$ is the family of all $k$-MDISs~$J$ such that $J \setminus \{\ell\} \in \mathcal{M}_k(T \setminus \ell)$. Clearly,  $|\mathcal{M}^1_k(T)| = \mathrm{mdi}_k(T \setminus \ell)$. Furthermore, for every $k$-MDIS $J' \in \mathcal{M}^2_k(T)$, the set $J' \setminus \{\ell\}$ is not inclusion-wise maximal in $T \setminus \ell$. Hence, there exists a vertex $u \in V(T \setminus \ell)$ such that the set $(J' \setminus \{\ell\}) \cup \{u\}$ is distance-$k$ independent in both $T \setminus \ell$ and $T$. Therefore, $N_k[u] \cap J' = \{\ell\}$ and $|\mathcal{M}^2_k(T)| = \mathrm{mdi}^*_k(T,
\ell)$, as required.
\end{proof}

\begin{lemma}\label{lem4}
Let $T$ be a tree with a leaf $u$ such that no leaf of $T$ belongs to $N_{k+1}[u] \setminus \{u\}$. Then, $\mathrm{mdi}_k(T \setminus u) < \mathrm{mdi}_k(T)$.
\end{lemma}

\begin{proof}
Let $v$ be the unique neighbor of $u$. By Lemma~\ref{lem3}, it suffices to show that there exists a set $J \in \mathcal{M}_k(T)$ such that $u \in J$ and $N_k[v] \cap J = \{u\}$. We construct such a set, starting with $J = \{u\}$. If there exists a vertex $x \in N_{k}(v)$ not covered by $J$, we add a neighbor $x' \in N_{k+1}(v)$ to $J$ (since $x$ is not a leaf, such a neighbor exists). Suppose that $x'$ is already covered by $J$. Then, there exists a vertex $y' \in N_{k+1}(v) \cap J$ such that $\mathrm{dist}(x',y') \leq k$. However, the only $x'y'$-path in $T$ passes through $x$; hence, $\mathrm{dist}(x,y') < k$ and $x$ is covered by $J$, a contradiction. Therefore, we can iteratively add vertices from $N_{k+1}(v)$ to $J$ until all vertices in $N_{k}(v)$ become covered. After that, we repeatedly add uncovered vertices from $V(T) \setminus N_k[v]$ to $J$ until it becomes a $k$-MDIS, as required.
\end{proof}

\begin{lemma}\label{lem5}
For all $k \geq 2$, the following hold:

(i) $\mathrm{mdi}_k(P_{k+2}) = k + 1$; 

(ii) $\mathrm{mdi}_k(P_{k+3}) = k + 2$; 

(iii) $\mathrm{mdi}_k(P_{k+4}) = k + 4$;

(iv) $\mathrm{mdi}_k(P_{k+5}) \geq k + 6$.
\end{lemma}

\begin{proof}
Statement (i) follows immediately from the fact that every $k$-MDIS of $P_{k+2}$ contains either both leaves or exactly one non-leaf vertex. To prove (ii), observe that for a leaf $\ell$ of $P_{k + 3}$, we have $\mathrm{mdi}_k^*(P_{k + 3},\ell) = 1$. By Lemma~\ref{lem3}, $\mathrm{mdi}_k(P_{k+3}) = \mathrm{mdi}_k(P_{k+2}) + 1$, as desired. Statements (iii) and (iv) can be proved similarly.
\end{proof}

\begin{lemma}\label{lem6}
Let $k \geq 2$, and let $T$ be a tree. Then, for a $k$-twin $u \in V(T)$, the following hold:

(i) If $u$ is a leaf, then $\mathrm{mdi}_k(T \setminus u) \leq \mathrm{mdi}_k(T) - 1$. Moreover, if $T$ has a unique $k$-MDIS containing $u$, then $\mathrm{mdi}_k(T \setminus u) = \mathrm{mdi}_k(T) - 1$.

(ii) If $u$ is a diametral leaf and $\mathrm{diam}(T) \geq k + 2$, then $\mathrm{mdi}_k(T \setminus u) \leq \mathrm{mdi}_k(T) - 2$.
\end{lemma}

\begin{proof}
Let $v \in V(T)$ be a vertex such that $N_k[v] = N_k[u]$. For every $k$-MDIS $J \ni u$, we have $N_k[v] \cap J = \{u\}$, and thus $\mathrm{mdi}_k^*(T,u) \geq 1$. Therefore, statement (i) follows from Lemma~\ref{lem3} and the fact that $\mathrm{mdi}_k^*(T,u)$ is bounded above by the total number of $k$-MDISs containing $u$. 

We now prove (ii). Let $P$ be a diametral path with endvertex $u$, and let $w$ denote its other endvertex. The condition $\mathrm{dist}(u,v) \leq k < \mathrm{diam}(T)$ implies that $v$ does not belong to $P$. Let $x$ be the vertex on $P$ that is farthest from $u$ among those on the $uv$-path. The condition $N_k[u] = N_k[v]$ implies $\mathrm{dist}(u,x) = \mathrm{dist}(v,x)$; hence, $v$ is also a diametral leaf of $T$. Let $w'$ be the unique neighbor of $w$.  There exist distinct $k$-MDISs $J$ and $J'$ such that  $\{u,w\} \subseteq J$ and $\{u,w'\} \subseteq J'$. Clearly, $N_k[v] \cap J = N_k[v] \cap J' = \{u\}$. Therefore, $\mathrm{mdi}_k^*(T,u) \geq 2$. Applying Lemma~\ref{lem3} yields the desired inequality.
\end{proof}

\begin{lemma}\label{lem7}
Let $T$ be a tree with a branching vertex $w$ adjacent to two pendant $s$-paths, where $s \geq 1$. Let $T'$ be the tree obtained from $T$ by removing one of these $s$-paths. Then, for all $k \geq 2$, the following hold:

(i) If $k \geq 2s$, then $\mathrm{mdi}_k(T) \geq \mathrm{mdi}_k(T') + s$.

(ii) If $k \in \{2s - 1,2s-2\}$, then $\mathrm{mdi}_k(T) \geq \mathrm{mdi}_k(T') + s - 1$.
\end{lemma}

\begin{proof}
Let $u_1\dots u_s$ and $v_1\dots v_s$ be the pendant paths adjacent to $w$ such that $u_1$ and $v_1$ are leaves of $T$, and suppose the path $v_1\dots v_s$ does not belong to $T'$. For $i \in [s]$, let $T_i$ denote the tree obtained from $T$ by removing the path $v_1\dots v_i$. Statement~(i) follows immediately from Lemma~\ref{lem6}(i) and the observation that if $k \geq 2s$, then $T \succ_k T_1 \succ_k \dots \succ_k T_s = T'$. 

We now prove (ii). By Lemma~\ref{lem3}, $\mathrm{mdi}_k(T) \geq \mathrm{mdi}_k(T_1)$. If $k = 2s - 1$, then the statement follows from the fact that $T_1 \succ_k T_2 \succ_k \dots \succ_k T_s = T'$. Suppose that $k = 2s - 2$. We first prove that $\mathrm{mdi}_k(T) > \mathrm{mdi}_k(T_2)$. Consider a mapping $F: \mathcal{M}_k(T_2) \rightarrow \mathcal{M}_k(T)$ defined by $F(J) = J$ if $J \in \mathcal{M}_k(T)$, and $F(J) = J \cup \{v_1\}$ otherwise. Since there exists a set $J' \in \mathcal{M}_k(T)$ containing $v_2$, $F$ is not a bijection; thus $\mathrm{mdi}_k(T) > \mathrm{mdi}_k(T_2)$. If $s = 2$, this immediately implies statement~(ii). If $s \geq 3$, statement~(ii) follows from the fact that $T_2 \succ_k T_3 \succ_k \dots \succ_k T_s = T'$, which, by Lemma~\ref{lem6}(i), implies that $\mathrm{mdi}_k(T') \leq \mathrm{mdi}_k(T_2) - s + 2$.
\end{proof}

\begin{lemma}\label{lem8}
For every $k \geq 2$, the following holds: if a tree $T$ of diameter at most $k + 3 - (k \bmod 2)$ has a non-central branching vertex, then it has a $k$-twin leaf.
\end{lemma}

\begin{proof}
Let $w$ be a non-central branching vertex of $T$, and let $u$ be the central vertex closest to $w$. Assume that $w$ is chosen so that $\mathrm{dist}(w,u)$ is as large as possible. Then~$w$ must be adjacent to at least two pendant paths not containing $u$, say $P'$ and $P''$, with leaves $\ell_1$ and $\ell_2$, respectively. Since $w$ is not a central vertex, $\max(\mathrm{dist}(\ell_1,w),\mathrm{dist}(\ell_2,w)) \leq \lfloor k/2 \rfloor$. We may assume that $\mathrm{dist}(\ell_1,w) \leq \mathrm{dist}(\ell_2,w)$; then there exists a vertex $x$ on $P''$ such that $N_k[\ell_1] = N_k[x]$. Thus, $\ell_1$ is a $k$-twin, as desired.
\end{proof}

\begin{lemma}\label{lem9}
For every odd $k \geq 3$, the following hold:

(i) A tree $T$ of diameter $k + 1$ has no $k$-twin leaves if and only if there exists $p \geq 2$ such that $T \cong S_{p,\lfloor k/2 \rfloor+1}$.

(ii) A tree $T$ of diameter $k + 2$ has no $k$-twin leaves if and only if there exists $p \geq 2$ such that $T \in \mathcal{B}_{p,\lfloor k/2 \rfloor+1}$.
\end{lemma}

\begin{proof}
We prove only (ii); statement (i) can be proved similarly. A tree in $\mathcal{B}_{p,\lfloor k/2 \rfloor+1}$ obviously has no $k$-twin leaves. Suppose that a tree $T \notin \mathcal{B}_{p,\lfloor k/2 \rfloor+1}$ of diameter $k + 2$ has no $k$-twin leaves. Let $u$ and $v$ be the central vertices of $T$. If either of them is adjacent to a pendant $s$-path with $s \leq \lfloor k/2 \rfloor$, then it is easy to check that the leaf of this path is a $k$-twin in $T$, a contradiction. Otherwise, there must exist a non-central branching vertex, which by Lemma~\ref{lem8}, implies that $T$ has a $k$-twin leaf, again a contradiction.
\end{proof}

\begin{lemma}\label{lem10}
Let $k,p \geq 2$, and let $T$ be a tree with a central leaf $\ell$ such that the main subtree $T_{\ell}$ containing $\ell$ has the maximum possible number of vertices among all main subtrees of $T$ that contain a central leaf. The following hold:

(i) If $T \in \mathrm{Add}_k(S'_{p,\lfloor k/2 \rfloor+1},r)$, $r \in \big[\lfloor k/2 \rfloor\big]$, then $T \setminus \ell \in \mathrm{Add}_k(S'_{p,\lfloor k/2 \rfloor+1},r - 1)$;

(ii) If $T \in \mathrm{Add}^*_k(S_{p,\lfloor k/2 \rfloor+1},r)$, $r \in \big[\lfloor k/2 \rfloor\big]$, then $T \setminus \ell \in \mathrm{Add}^*_k(S_{p,\lfloor k/2 \rfloor+1},r - 1)$.
\end{lemma}

\begin{proof}
We prove only (i); statement (ii) can be proved similarly. Since $T$ is not a simple path, it follows from the choice of $T_{\ell}$ that $\mathrm{diam}(T \setminus \ell) = \mathrm{diam}(T)$. By definition of $\mathrm{Add}_k(S'_{p,\lfloor k/2 \rfloor+1},r)$, $T$ does not have diametral $k$-twins; since removing $\ell$ cannot create any diametral $k$-twins, $T \setminus \ell$ also lacks them. It remains to show that $S'_{p,\lfloor k/2 \rfloor+1} \subseteq T \setminus \ell$.  If $p = 2$, this is trivially true because $S'_{2,\lfloor k/2 \rfloor+1} \cong P_{2\lfloor k/2 \rfloor+2}$. If $p \geq 3$, then~$T$ has exactly one central vertex $v$; moreover, this vertex is also central in every $S'_{p,\lfloor k/2 \rfloor+1}$-subtree of $T$. We now consider three cases: 

\textbf{Case 1.} $\mathrm{ecc}(\ell) \leq 2\lfloor k/2 \rfloor$. Then $\mathrm{dist}(v,\ell) < \lfloor k/2 \rfloor$, which means $\ell$ cannot be a leaf in any $S'_{p,\lfloor k/2 \rfloor+1}$-subtree of $T$. Therefore, $S'_{p,\lfloor k/2 \rfloor+1} \subseteq T \setminus \ell$. 

\textbf{Case 2.} $\mathrm{ecc}(\ell) = 2\lfloor k/2 \rfloor + 1$.

\textbf{Case 2.1.} $T_{\ell}$ is a pendant $\lfloor k/2 \rfloor$-path. Recall that $T$ does not contain diametral $k$-twins. Then, by the choice of $T_{\ell}$, all other main subtrees of $T$ are pendant $\lfloor k/2 \rfloor$- or $(\lfloor k/2 \rfloor + 1)$-paths. Since $p \geq 3$, $T$ has at least two pendant $\lfloor k/2 \rfloor$-paths, implying $S'_{p,\lfloor k/2 \rfloor+1} \subseteq (T \setminus T_{\ell}) \subseteq T \setminus \ell$. 

\textbf{Case 2.2.} $T_{\ell}$ is not a pendant $\lfloor k/2 \rfloor$-path. Let~$\ell'$ denote the leaf of maximum eccentricity in $T_{\ell}$; we may assume $\ell' \neq \ell$. If there exists an $S'_{p,\lfloor k/2 \rfloor+1}$-subtree of $T$ containing $\ell'$, then this subtree does not contain $\ell$, as desired. Otherwise, no $S'_{p,\lfloor k/2 \rfloor+1}$-subtree contains vertices from $T_{\ell}$, implying $S'_{p,\lfloor k/2 \rfloor+1} \subseteq T \setminus T_{\ell} \subseteq T \setminus \ell$.

\textbf{Case 3.} $\mathrm{ecc}(\ell) = 2\lfloor k/2 \rfloor + 2$. That is, $\ell$ is a diametral leaf, but not a $k$-twin. By the choice of $T_{\ell}$, all main subtrees of $T$ are pendant $(\lfloor k/2 \rfloor+1)$-paths, meaning $T \cong S_{p,\lfloor k/2 \rfloor+1} \in \mathrm{Add}_k(S'_{p,\lfloor k/2 \rfloor+1},1)$ and $T \setminus \ell \cong S'_{p,\lfloor k/2 \rfloor+1} \in \mathrm{Add}_k(S'_{p,\lfloor k/2 \rfloor+1},0)$, as required. 
\end{proof}

\section{Main result}

Our proof relies on the following key properties:

\begin{itemize}
	\item Removing a leaf from a tree cannot increase the total number of $k$-MDISs (Lemma~\ref{lem3}). Furthermore, if the removed leaf was a $k$-twin, the total number of $k$-MDISs strictly decreases  (Lemma~\ref{lem6}).
	\item Trees of small diameter with no $k$-twin leaves have a simple structure (Lemmas~\ref{lem8} and~\ref{lem9}).
\end{itemize}

The proof strategy differs a bit depending on the parity of $k$; this is why we consider the even case and the odd case separately. If $k$ is even, then all $k$-minimal trees with at least $(3k+6)/2$ vertices have diameter $k + 2$. However, if $k$ is odd, then large $k$-minimal trees may have diameter $k + 1$ or $k + 2$. 

\subsection{Even $k \geq 2$}

\begin{lemma}\label{lem11}
Let $T$ be an $n$-vertex tree of diameter $k + 1$. Then $\mathrm{mdi}_k(T) \geq n - 1$ with equality if and only if $T \in \mathrm{Add}^*_k(P_{k+2},n-k-2)$.
\end{lemma}

\begin{proof}
Let $T_1$ and $T_2$ denote the inclusion-wise maximal subtrees obtained from~$T$ by deleting the edge between its central vertices. Furthermore, let $\ell_1$ and~$\ell_2$ denote the number of diametral leaves of $T$ belonging to $T_1$ and $T_2$, respectively. Every $k$-MDIS of $T$ is either a singleton containing a vertex that is not a diametral leaf or a pair of vertices consisting of one diametral leaf from $T_1$ and one from $T_2$. Therefore, $$\mathrm{mdi}_k(T) = (n - \ell_1 - \ell_2) + (\ell_1 \cdot \ell_2) \geq n - 1.$$
Equality holds if and only if $\min(\ell_1,\ell_2) = 1$, which in turn holds if and only if $T$ has no $k$-special pairs. By definition, this means $T \in \mathrm{Add}^*_k(P_{k+2},n-k-2)$.
\end{proof}

\begin{lemma}\label{lem12}
Let $T$ be an $n$-vertex tree of diameter at most $k + 2$. Then the following hold:

(i) If $n = p \cdot (k/2 + 1)$, where $p \geq 2$, then $\mathrm{mdi}_k(T) \geq f_k(n)$ with equality if and only if $T \cong S'_{p,k/2+1}$.

(ii) If $n = 2 \cdot (k/2 + 1) + r$, where $r \in [k/2]$, then $\mathrm{mdi}_k(T) \geq  f_k(n)$ with equality if and only if $T \in \mathrm{Add}_k^*(P_{k+2},r) \cup \mathrm{Add}_k(P_{k+3},r-1)$.

(iii) If $n = p \cdot (k/2 + 1) + r$, where $p \geq 3$ and $r \in [k/2]$, then $\mathrm{mdi}_k(T) \geq  f_k(n)$ with equality if and only if $T \in \mathrm{Add}_k(S'_{p,k/2+1},r)$.
\end{lemma}

\begin{proof} 

\textbf{Statement (i).} We apply induction on $p$; the base case $p = 2$ follows from Lemma~\ref{lem11}. For $p \geq 3$, let $T \ncong S'_{p,k/2+1}$ be a $p \cdot (k/2 + 1)$-vertex tree of diameter at most $k + 2$. Since $f_k(p \cdot (k/2 + 1)) < n - 1$, we may assume that $T$ has diameter $k + 2$; let $v$ denote its central vertex. Recall that $f_k(n) = f_k(n - k/2 - 1) + k/2$ by Lemma~\ref{lem1}(ii). To prove that $\mathrm{mdi}_k(T) > f_k(n)$, it suffices to show that either $T$ is $(k/2+1)$-reducible or it is $(k/2)$-reducible to the tree $S_{p-1,k/2+1}$ (Lemma~\ref{lem2} implies that $\mathrm{mdi}_k(S'_{p-1,k/2+1}) < \mathrm{mdi}_k(S_{p-1,k/2+1})$ for all $p \geq 3$ and even $k \geq 2$).

\textbf{Case 1.} $T$ has no branching vertices except $v$. Let $s = \deg(v)$, and let $m_1,\dots,m_s$ denote the number of vertices in the pendant paths adjacent to $v$. Since $\mathrm{diam}(T) = k + 2$, we may assume that $k/2 + 1 = m_1 = m_2 \geq \dots \geq m_s \geq 1$. Let $i$ be the smallest index such that $m_i \leq k/2$ (it exists because $k/2 + 1$ does not divide $n - 1$). If $i = s$, then $m_s = k/2$ and $T \cong S'_{p,k/2+1}$, a contradiction. If $i < s$, consider two subcases:

\textbf{Case 1.1.} $m_s = k/2$.  Note that $k/2+1$ divides $s - i$; hence, $T$ has at least $k/2+2$ pendant $(k/2)$-paths. Let $\ell_1,\dots,\ell_{k/2 + 2}$ be the endvertices of $k/2+2$ such paths. For all $j \in [k/2+1]$, let $T_j$ be the tree obtained from $T$ by deleting the leaves $\ell_1,\dots,\ell_j$. Then $T$ is $(k/2+1)$-reducible, as $T \succ_k T_1 \succ_k \dots \succ_k T_{k/2+1}$.

\textbf{Case 1.2.} $m_s < k/2$. Let $M = \sum_{i' = i}^sm_{i'}$. Since $k/2 + 1$ divides $n$, we have $M \geq k/2$. Let $T_0 = T$, and for all $j \in [k/2]$, choose a tree $T_j$ obtained from $T_{j-1}$ by removing an endvertex of a shortest path of $T_{j-1}$. Observe that, at each step, the removed leaf is a $k$-twin; thus, $T \succ_k T_1 \succ_k \dots \succ_k T_{k/2}$. If, moreover, $T_{k/2} \ncong S_{p-1,k/2+1}$, then the endvertex of a shortest path of $T_{k/2}$ is again a $k$-twin; hence, $T$ is $(k/2+1)$-reducible.

\textbf{Case 2.} $T$ has a non-central branching vertex. Suppose $T$ is not $(k/2)$-reducible to $S_{p-1,k/2+1}$. By Lemma~\ref{lem8}, $T$ has a $k$-twin leaf; hence, it is reducible to some tree $T_1$. If $T_1$ also has a non-central branching vertex, then it also has a $k$-twin. Otherwise, since $k/2+1$ does not divide $|V(T_1)|$, there exist at least two pendant paths with fewer than $k/2+1$ vertices, or a pendant path with fewer than $k/2$ vertices. Therefore, a central leaf of $T_1$ is a $k$-twin, and $T_1$ is reducible to some tree $T_2$. If $k = 2$, $T$ is $2$-reducible, as desired. Otherwise, one can apply this argument $k/2-1$ more times to sequentially obtain trees $T_3,\dots, T_{k/2+1}$. Hence, $T$ is $(k/2+1)$-reducible.

\textbf{Statement (ii).} We apply induction on $r$. The base case $r = 1$ follows from Lemmas~\ref{lem11} and~\ref{lem5}(ii).

\textbf{Case 1.} $T \in \mathrm{Add}_k^*(P_{k+2},r) \cup \mathrm{Add}_k(P_{k+3},r-1)$. If $T \in \mathrm{Add}_k^*(P_{k+2},r)$, then $\mathrm{mdi}_k(T) = f_k(n)$ by Lemma~\ref{lem11}. Suppose that $T \in \mathrm{Add}_k(P_{k+3},r-1)$ and $r \geq 2$. One can see that there exists a central leaf $\ell$ of $T$ such that $\mathrm{diam}(T \setminus \ell) = \mathrm{diam}(T)$ and, therefore, $T \setminus \ell \in \mathrm{Add}_k(P_{k+3},r - 2)$. If $\mathrm{ecc}(\ell) \leq k + 1$, then $\mathrm{mdi}^*_k(T,\ell) = 1$, as there is exactly one $k$-MDIS containing $\ell$; hence, $\mathrm{mdi}_k(T) = \mathrm{mdi}_k(T \setminus \ell) + 1$ by Lemma~\ref{lem6}(i), and thus $\mathrm{mdi}_k(T) = f_k(n-1) + 1 = f_k(n)$. If $\mathrm{ecc}(\ell) = k + 2$, then $T \cong P_{k+3}$, because $T$ has no diametral $k$-twins, a contradiction.

\textbf{Case 2.} $T \notin \mathrm{Add}_k^*(P_{k+2},r) \cup \mathrm{Add}_k(P_{k+3},r-1)$.

\textbf{Case 2.1.} $\mathrm{diam}(T) \leq k + 1$. If $\mathrm{diam}(T) \leq k$, then $\mathrm{mdi}_k(T) = n  > f_k(n)$. Suppose that $\mathrm{diam}(T) = k + 1$. The condition $T \notin \mathrm{Add}_k^*(P_{k+2},r)$ implies that $T$ has a $k$-special pair. Then $\mathrm{mdi}_k(T) \geq n > f_k(n)$ by the previous lemma.

\textbf{Case 2.2.} $\mathrm{diam}(T) = k + 2$. The condition $T \notin \mathrm{Add}_k(P_{k+3},r-1)$ implies that $T$ has a diametral $k$-twin $\ell$. Lemma~\ref{lem6}(ii) and the inductive hypothesis imply $$\mathrm{mdi}_k(T) \geq \mathrm{mdi}_k(T \setminus \ell) + 2 \geq f_k(n-1) + 2 > f_k(n).$$

\textbf{Statement (iii).} For a fixed $p \geq 3$, we apply induction on $r$; the base case $r = 0$ follows from statement~(i), as $\mathrm{Add}_k(S'_{p,k/2+1},0) = \{S'_{p,k/2+1}\}$. Since $f_k(n) < n - 1$, Lemma~\ref{lem11} implies that $\mathrm{diam}(T) = k + 2$. Let $v$ be the central vertex of $T$.

\textbf{Case 1.}  $T \notin \mathrm{Add}_k(S'_{p,k/2+1},r)$. There are two possibilities:

\textbf{Case 1.1.} $S'_{p,k/2+1} \nsubseteq T$. If $T$ has no branching vertices except $v$, consider its central leaf $\ell$. Clearly, $S'_{p,k/2+1} \nsubseteq T \setminus \ell$; thus, $T \setminus \ell \notin \mathrm{Add}_k(S'_{p,k/2+1},r-1)$. Moreover, either $\mathrm{ecc}(\ell) \leq k$ and $N_k[\ell] = N_k[v]$, or $\mathrm{ecc}(\ell) = k + 1$ and there exists another leaf $\ell'$ with $\mathrm{ecc}(\ell') = k + 1$ and $N_k[\ell] = N_k[\ell']$.  Therefore, $\ell$ is a $k$-twin; the inductive hypothesis and Lemma~\ref{lem6}(i) imply $$\mathrm{mdi}_k(T) \geq \mathrm{mdi}_k(T \setminus \ell)  + 1 > f_k(n-1) + 1 = f_k(n).$$ 

If $T$ has a non-central branching vertex, then it has a $k$-twin leaf by Lemma~\ref{lem8}, and we can apply the same argument. 

\textbf{Case 1.2.} $S'_{p,k/2+1} \subseteq T$. Then, since $T \notin \mathrm{Add}_k(S'_{p,k/2+1},r)$, there exists a diametral $k$-twin $\ell \in V(T)$. Lemma~\ref{lem6}(ii) implies that $$\mathrm{mdi}_k(T) \geq \mathrm{mdi}_k(T \setminus \ell) + 2 \geq f_k(n-1) + 2 > f_k(n).$$

\textbf{Case 2.} $T \in \mathrm{Add}_k(S'_{p,k/2+1},r)$. By Lemma~\ref{lem10}(i), there exists a central leaf $\ell$ of $T$ such that $T \setminus \ell \in \mathrm{Add}_k(S'_{p,k/2+1},r - 1)$. If $\mathrm{ecc}(\ell) \leq k + 1$, then $\mathrm{mdi}^*_k(T,\ell) = 1$, as there is exactly one $k$-MDIS containing $\ell$; hence, $\mathrm{mdi}_k(T) = \mathrm{mdi}_k(T \setminus \ell) + 1$ by Lemma~\ref{lem6}(i) and $\mathrm{mdi}_k(T) = f_k(n-1) + 1 = f_k(n)$. If $\mathrm{ecc}(\ell) = k + 2$, then $T \cong S_{p,k/2+1}$, because $T$ has no diametral $k$-twins, and the statement holds by Lemma~\ref{lem2}.
\end{proof}

\begin{lemma}\label{lem13}
Let $T$ be an $n$-vertex tree of diameter $k + 3$. Then $\mathrm{mdi}_k(T) > f_k(n)$.
\end{lemma}
\begin{proof}
Let $T$ be an $n$-vertex tree of diameter $k + 3$ with central vertices $u$ and $v$. We apply induction on $n$; the base case $n = k + 4$ holds by Lemma~\ref{lem5}(iii). By Lemma~\ref{lem1}(i), it suffices to prove that there exists an $(n-1)$-vertex tree $T'$ such that $\mathrm{mdi}_k(T) > \mathrm{mdi}_k(T')$ and $\mathrm{mdi}_k(T')  > f_k(n-1)$.

\textbf{Case 1.} $T$ has no non-central  branching  vertices. That is, $u$ and $v$ are adjacent to $\deg(u) - 1$ and $\deg(v) - 1$ simple paths, respectively. Let $s \geq 1$ be the smallest number of vertices in any of these paths. Since $T$ itself is not a simple path, we may assume that $u$ is adjacent to a pendant $s$-path and $\deg(u) \geq 3$.

\textbf{Case 1.1.} $s < k/2$. Let $x$ be an endvertex of a pendant $s$-path adjacent to $u$. Since $u$ is a central vertex of $T$, it is also adjacent to a pendant $(k/2 + 1)$-path containing a vertex $y$ such that $\mathrm{dist}(x,u) = \mathrm{dist}(y,u)$. Then $x$ is a $k$-twin and, by Lemma~\ref{lem6}(i), $\mathrm{mdi}_k(T) > \mathrm{mdi}_k(T \setminus x)$. Note that $\mathrm{diam}(T) = \mathrm{diam}(T \setminus x)$; hence, by the inductive hypothesis, $\mathrm{mdi}_k(T \setminus x) > f_k(n-1)$ and the statement follows.

\textbf{Case 1.2.} $s = k/2$ and $v$ is adjacent to a pendant $(k/2)$-path. We may assume that $\deg(u) \geq \deg(v) \geq 3$.  Suppose that $\deg(u) \geq 4$. If $u$ is adjacent to two distinct pendant $(k/2)$-paths, we can apply the same argument as in Case~1.1. Otherwise, $u$ is adjacent to two distinct pendant $(k/2 + 1)$-paths; let $T^*$ be the tree obtained by removing one of them. By the inductive hypothesis, $\mathrm{mdi}_k(T^*) > f_k(n - k/2 - 1)$. Furthermore, $\mathrm{mdi}_k(T) \geq \mathrm{mdi}_k(T^*) + k/2$ by Lemma~\ref{lem7}, and hence $\mathrm{mdi}_k(T) > f_k(n)$ by Lemma~\ref{lem1}(ii). If $\deg(u) = \deg(v) = 3$, then $|V(T)| = 2k + 4$ and a simple calculation shows that $\mathrm{mdi}_k(T) > 2k + 1= f_k(2k + 4)$.
 
\textbf{Case 1.3.} $s = k/2$ and $v$ is not adjacent to a pendant $(k/2)$-path. If $\deg(u) \geq 4$ or $\deg(v) \geq 3$, we can apply the same argument as in Case~1.2. If $\deg(u) = 3$ and $\deg(v) = 2$, then $|V(T)| = 3k/2 + 4$ and a simple calculation shows that $\mathrm{mdi}_k(T)  > 3k/2 + 2= f_k(3k/2+4)$.

\textbf{Case 1.4.} $s = k/2 + 1$. Since $T$ has no non-central branching vertices, the only possible scenario is $T \cong B_{p_1,p_2,k/2 + 1}$ for some $p_1 \geq 2$ and $p_2 \geq 1$. Hence, the statement follows from Lemma~\ref{lem2}(iii).

\textbf{Case 2.} $T$ has a  non-central branching vertex. By Lemma~\ref{lem8}, it has a $k$-twin leaf $x$. By Lemma~\ref{lem6}(i), $\mathrm{mdi}_k(T) \geq \mathrm{mdi}_k(T \setminus x) + 1$. Since $x$ is a $k$-twin, one can see that $T$ has a diametral path not containing $x$; hence,  $\mathrm{diam}(T \setminus x) = \mathrm{diam}(T)$. By the inductive hypothesis, $\mathrm{mdi}_k(T \setminus x) > f_k(n-1)$; hence, $\mathrm{mdi}_k(T) > f_k(n)$.
\end{proof}

\begin{lemma}\label{lem14}
Let $T$ be an $n$-vertex tree of diameter at least $k + 4$. Then $\mathrm{mdi}_k(T) > f_k(n)$.
\end{lemma}

\begin{proof}
We apply induction on $n$. The base case holds by Lemma~\ref{lem5}(iv).  Consider a branching vertex $w \in V(T)$ adjacent to pendant paths $P_a$ and $P_b$ (with $a \leq b$) having endvertices $a_1$ and $b_1$, with respective neighbors $a_2$ and $b_2$ (if $a = 1$ or $b = 1$, the corresponding neighbor coincides with~$w$). We may assume that $w$ is chosen in such a way that the parameter $a$ is minimal. The following cases are possible:

\textbf{Case~1. $a + b \leq k$.} Then there exists a vertex $x$ on $P_b$ such that $N_k[x] = N_k[a_1]$. Note that $\mathrm{diam}(T \setminus a_1) \geq k + 3$. The inductive hypothesis and Lemma~\ref{lem6}(i) imply $$\mathrm{mdi}_k(T) \geq \mathrm{mdi}_k(T \setminus a_1) + 1 > f_k(n - 1) + 1 \geq f_k(n).$$

\textbf{Case~2. $a \in \{k/2,k/2+1\}$, $b = k/2 + 1$.} Consider the tree $T'$ obtained from $T$ by removing $P_b$. If $a = k/2$, then, by Lemma~\ref{lem7}(i), $\mathrm{mdi}_k(T') \leq \mathrm{mdi}_k(T \setminus b_1) - k/2$. Moreover, $\mathrm{diam}(T') \geq k + 3$; hence, $\mathrm{mdi}_k(T') > f_k(n - k/2 - 1) = f_k(n) - k/2$, which implies $\mathrm{mdi}_k(T) > f_k(n)$.  If $a = k/2 + 1$, then, by Lemma~\ref{lem7}(ii), $\mathrm{mdi}_k(T') \leq \mathrm{mdi}_k(T) - k/2$ which implies $\mathrm{mdi}_k(T) > f_k(n)$.

\textbf{Case~3. $b \geq k/2 + 2$.}  First, we show that there exists a vertex $x \in V(T)$ such that $\mathrm{dist}(b_1,x) = k + 2$. Suppose not; then $b_1$ does not belong to any diametral path of~$T$, and every vertex on every diametral path of~$T$ belongs to $N_{k/2}[w]$, which is impossible. Next, consider a $k$-MDIS $J$ containing $b_1$ and $x$. If $x$ belongs to $P_b$, then $N_k[b_2] \cap J = \{b_1\}$. Suppose $x$ does not belong to $P_b$ and there exists a vertex $y \in J \setminus \{b_1,x\}$ such that $y \in N_k[b_2]$. Then $y$ does not belong to the $b_1x$-path (in particular, it does not belong to $P_b$). Therefore, 
$$ \mathrm{dist}(x,y) \leq \mathrm{dist}(x,w) + \mathrm{dist}(y,w) \leq 2 \cdot \mathrm{dist}(x,w) = 2 \cdot (\mathrm{dist}(b_1,x) - \mathrm{dist}(b_1,w)) \leq k.$$

Hence, such a vertex $y$ does not exist, and $N_k[b_2] \cap J = \{b_1\}$, so $\mathrm{mdi}_k^*(T,b_1) \geq 1$ by Lemma~\ref{lem3}. Since $\mathrm{diam}(T \setminus b_1) \geq k + 3$, Lemma~\ref{lem13} and the inductive hypothesis imply $\mathrm{mdi}_k(T \setminus b_1) > f_k(n-1)$; thus, $\mathrm{mdi}_k(T) > f_k(n)$.
\end{proof}

Lemmas~\ref{lem11}--\ref{lem14} immediately imply the following fact:

\begin{theorem}\label{thm1}
For all even $k \geq 2$ and all $n \geq 1$, every $k$-minimal $n$-vertex tree has exactly $f_k(n)$ $k$-MDISs. Furthermore, the following hold:

(i) If $n \leq k + 1$, then all $n$-vertex trees are $k$-minimal; if $n = k + 2$, then the path $P_{k+2}$ is the only $k$-minimal $n$-vertex tree.

(ii) If $k + 3 \leq n \leq 3k/2 + 2$, then $$\mathrm{Add}^*_k(P_{k+2},n-k-2) \cup \mathrm{Add}_k(P_{k+3},n - k - 3)$$ is the class of all $k$-minimal $n$-vertex trees.

(iii) If $n \geq 3k/2 + 3$, then  $\mathrm{Add}_k(S'_{p,k/2 + 1},r)$ is the class of all $k$-minimal $n$-vertex trees, where $p = \lfloor n/(k/2 + 1) \rfloor$ and $r = n\bmod (k/2 + 1)$.
 
\end{theorem}

For $p \geq 2$, denote by $S''_{p,2}$ the tree obtained from $S'_{p,2}$ by adding one leaf to its central vertex. The next observation follows from Theorem~1.

\begin{corollary}
The following hold:

(i) For $n \in \{1,2,3\}$, the path $P_n$ is the only $n$-vertex $2$-minimal tree;

(ii) For all even $n \geq 4$, the tree $S'_{n/2,2}$ is the only $n$-vertex $2$-minimal tree;

(iii) For all odd $n \geq 5$, there exist two non-isomorphic $n$-vertex $2$-minimal trees: $S_{(n-1)/2,2}$ and $S''_{(n-1)/2,2}$.
\end{corollary}

Finally, we show that for all even $k$, the number of non-isomorphic $n$-vertex $k$-minimal trees is bounded above by a constant that depends only on $k$. Recall that $t(m)$ denotes the number of non-isomorphic unlabeled $m$-vertex trees.

\begin{theorem}\label{thm2}
For all $n \geq 2$ and all even $k \geq 2$, there exist at most $t(k^2)$ non-isomorphic $n$-vertex $k$-minimal trees.
\end{theorem}

\begin{proof}
If $k = 2$ or $n \leq k^2$, the statement holds trivially; let $k \geq 4$ and $n = p \cdot (k/2 + 1) + r$ for some $p \geq k + 1$ and $0 \leq r \leq k/2$. Our aim is to prove that $|\mathrm{Add}_k(S'_{p,k/2+1},r)| \leq t(k^2)$. Observe that every tree from $\mathrm{Add}_k(S'_{p,k/2+1},r)$ has a unique central vertex and at least $p - r - 1 \geq 2$ pendant $(k/2+1)$-paths connected to it. Consider a function $$F: \mathrm{Add}_k(S'_{p,k/2+1},r) \longrightarrow \mathrm{Add}_k(S'_{r + 3,k/2+1},r)$$ that maps a tree $T \in \mathrm{Add}_k(S'_{p,k/2+1},r)$ to a tree $F(T)$ with $p - r - 3$ fewer pendant $(k/2+1)$-paths. Observe that $T$ and $F(T)$ have the same central vertex; thus, it is easy to see that $F$ is an injection. If $(k,r) \neq (4,2)$, then $(r + 3) \cdot (k/2 + 1) + r < k^2$ and $$|\mathrm{Add}_k(S'_{p,k/2+1},r)| \leq |\mathrm{Add}_k(S'_{r + 3,k/2+1},r)| \leq t(k^2).$$
Finally, if $(k,r) = (4,2)$, then it is straightforward to check that the family $\mathrm{Add}_4(S'_{5,3},2)$ contains less than $t(16) = 19320$ non-isomorphic trees. This completes the proof.
\end{proof}

\begin{figure}[h!]
\center
\begin{tikzpicture}
  [scale=1,auto=left,every node/.style={circle}]

\node[fill = black,inner sep=2.5pt] (b11) at (1,1) {};
\node[fill = black,inner sep=2.5pt] (b28) at (0,0) {};
\node[fill = black,inner sep=2.5pt] (b38) at (0,-1) {};
\node[fill = black,inner sep=2.5pt] (b48) at (0,-2) {};
\node[fill = black,inner sep=2.5pt] (b29) at (0.5,0) {};
\node[fill = black,inner sep=2.5pt] (b39) at (0.5,-1) {};
\node[fill = black,inner sep=2.5pt] (b49) at (0.5,-2) {};
\node[fill = black,inner sep=2.5pt] (b20) at (1,0) {};
\node[fill = black,inner sep=2.5pt] (b30) at (1,-1) {};
\node[fill = black,inner sep=2.5pt] (b40) at (1,-2) {};
\node[fill = black,inner sep=2.5pt] (b21) at (1.5,0) {};
\node[fill = black,inner sep=2.5pt] (b31) at (1.5,-1) {};
\node[fill = black,inner sep=2.5pt] (b41) at (1.5,-2) {};
\node[fill = black,inner sep=2.5pt] (b22) at (2,0) {};
\node[fill = black,inner sep=2.5pt] (b32) at (2,-1) {};
\node[fill = gray,inner sep=1.75pt] (b10) at (2,0.5) {};
\node[fill = gray,inner sep=1.75pt] (b12) at (2,1) {};

\node[fill = black,inner sep=2.5pt] (c11) at (9,1) {};
\node[fill = black,inner sep=2.5pt] (c28) at (8,0) {};
\node[fill = black,inner sep=2.5pt] (c38) at (8,-1) {};
\node[fill = black,inner sep=2.5pt] (c48) at (8,-2) {};
\node[fill = black,inner sep=2.5pt] (c29) at (8.5,0) {};
\node[fill = black,inner sep=2.5pt] (c39) at (8.5,-1) {};
\node[fill = black,inner sep=2.5pt] (c49) at (8.5,-2) {};
\node[fill = black,inner sep=2.5pt] (c20) at (9,0) {};
\node[fill = black,inner sep=2.5pt] (c30) at (9,-1) {};
\node[fill = black,inner sep=2.5pt] (c40) at (9,-2) {};
\node[fill = black,inner sep=2.5pt] (c21) at (9.5,0) {};
\node[fill = black,inner sep=2.5pt] (c31) at (9.5,-1) {};
\node[fill = black,inner sep=2.5pt] (c41) at (9.5,-2) {};
\node[fill = gray,inner sep=1.75pt] (c12) at (11,-1) {};
\node[fill = black,inner sep=2.5pt] (c22) at (10.5,0) {};
\node[fill = black,inner sep=2.5pt] (c32) at (10.5,-1) {};
\node[fill = gray,inner sep=1.75pt] (c42) at (10,-1) {};

\node[fill = black,inner sep=2.5pt] (a11) at (5,1) {};
\node[fill = black,inner sep=2.5pt] (a28) at (4,0) {};
\node[fill = black,inner sep=2.5pt] (a38) at (4,-1) {};
\node[fill = black,inner sep=2.5pt] (a48) at (4,-2) {};
\node[fill = black,inner sep=2.5pt] (a29) at (4.5,0) {};
\node[fill = black,inner sep=2.5pt] (a39) at (4.5,-1) {};
\node[fill = black,inner sep=2.5pt] (a49) at (4.5,-2) {};
\node[fill = black,inner sep=2.5pt] (a20) at (5,0) {};
\node[fill = black,inner sep=2.5pt] (a30) at (5,-1) {};
\node[fill = black,inner sep=2.5pt] (a40) at (5,-2) {};
\node[fill = black,inner sep=2.5pt] (a21) at (5.5,0) {};
\node[fill = black,inner sep=2.5pt] (a31) at (5.5,-1) {};
\node[fill = black,inner sep=2.5pt] (a41) at (5.5,-2) {};
\node[fill = black,inner sep=2.5pt] (a22) at (6,0) {};
\node[fill = black,inner sep=2.5pt] (a32) at (6,-1) {};
\node[fill = gray,inner sep=1.75pt] (a23) at (6.5,0) {};
\node[fill = gray,inner sep=1.75pt] (a33) at (6.5,-1) {};

\foreach \from/\to in {b11/b28,b28/b38,b38/b48, b11/b29, b29/b39,b39/b49,b11/b20,b20/b30,b30/b40,b11/b21,b21/b31,b31/b41,b11/b22,b22/b32,b11/b10,b11/b12,c11/c28,c28/c38,c38/c48, c11/c29, c29/c39, c39/c49, c11/c20,c20/c30,c30/c40,c12/c22,c21/c31,c31/c41,c11/c22,c11/c21,c22/c32,c21/c42,a11/a28,a28/a38,a38/a48,a11/a29,a29/a39,a39/a49,a11/a20,a20/a30,a30/a40,a11/a21, a21/a31, a31/a41, a11/a22, a22/a32, a11/a23, a23/a33}
    \draw[thick] (\from) -- (\to);

\end{tikzpicture}

\caption{\small  Examples of trees from the family $\mathrm{Add}_4(S'_{5,3},2)$.}
\end{figure}
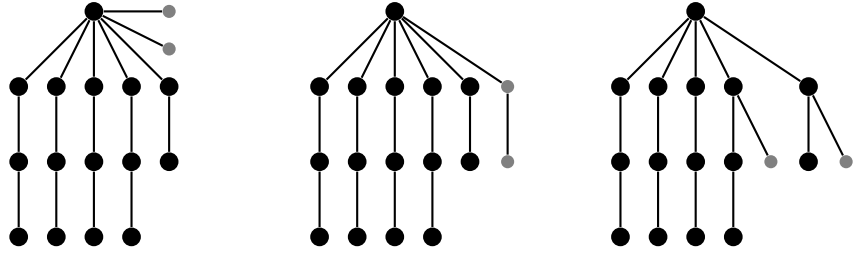

\subsection{Odd $k \geq 3$}

Recall that every $n$-vertex tree with diameter at most $k$ (in particular, with at most $k + 1$ vertices) has exactly $n$ $k$-MDISs.

\begin{lemma}\label{lem15}
Let $T$ be an $n$-vertex tree of diameter $k + 1$ with $n = p \cdot (\lfloor k/2 \rfloor+1) + r + 1$, where $p \geq 2$ and $0 \leq r \leq \lfloor k/2 \rfloor$. Then $\mathrm{mdi}_k(T) \geq f_k(n)$ with equality if and only if $T \in \mathrm{Add}_k^*(S_{p,\lfloor k/2 \rfloor+1},r)$.
\end{lemma}

\begin{proof}
We apply induction on $n$. The base case $n = k + 2$ holds trivially, as the only $(k + 2)$-vertex tree of diameter at least $k + 1$ is isomorphic to $S_{2,\lfloor k/2 \rfloor+1}$. For the inductive step, consider two cases:

\textbf{Case 1.} $r = 0$. Recall that $\mathrm{Add}_k^*(S_{p,\lfloor k/2 \rfloor+1},0) = \{S_{p,\lfloor k/2 \rfloor+1}\}$ and, moreover, $\mathrm{mdi}_k(S_{p,\lfloor k/2 \rfloor+1}) = f_k(n)$. Consider a tree $T \ncong S_{p,\lfloor k/2 \rfloor + 1}$. Lemma~\ref{lem9}(i) implies that $T$ is $(\lfloor k/2 \rfloor+1)$-reducible to some tree $T^*$.  Lemma~\ref{lem6}(i) and the inductive hypothesis imply $$\mathrm{mdi}_k(T) \geq \mathrm{mdi}_k(T^*) + \lfloor k/2 \rfloor + 1 \geq f_k(n - \lfloor k/2 \rfloor - 1) + \lfloor k/2 \rfloor + 1 > f_k(n).$$

\textbf{Case 2.} $1 \leq r \leq \lfloor k/2 \rfloor$.

\textbf{Case 2.1.} $T \in \mathrm{Add}_k^*(S_{p,\lfloor k/2 \rfloor+1},r)$. Consider a central leaf $\ell \in V(T)$ such that the main subtree $T_{\ell}$ containing $\ell$ has the maximum possible number of vertices. Since $r > 0$, $T_{\ell}$ is not a pendant $(\lfloor k/2 \rfloor+1)$-path; thus, $\ell$ is a $k$-twin in $T$. Furthermore, by Lemma~\ref{lem10}(ii), $T \setminus \ell \in \mathrm{Add}_k^*(S_{p,\lfloor k/2 \rfloor+1},r-1)$. It remains to prove that $\mathrm{mdi}^*_k(T,\ell) = 1$. If $\ell$ is not a diametral leaf, the singleton $\{\ell\}$ is the unique $k$-MDIS containing~$\ell$, and the equality follows from Lemma~\ref{lem6}(i). If $\ell$ is a diametral leaf, then because $T$ contains no $k$-special pairs, each of its main subtrees except possibly $T_{\ell}$ contains at most one diametral leaf. Thus, there exists a unique $k$-MDIS containing $\ell$ (which also contains all diametral leaves of $T$ not in $T_{\ell}$), and the equality $\mathrm{mdi}^*_k(T,\ell) = 1$ follows from Lemma~\ref{lem6}(i).

\textbf{Case 2.2.} $T \notin \mathrm{Add}_k^*(S_{p,\lfloor k/2 \rfloor+1},r)$. If $T$ has a $k$-special pair $(x,y)$, then \mbox{$\mathrm{mdi}^*_k(T,x) \geq 2$}; hence, Lemma~\ref{lem3} and the inductive hypothesis imply $$\mathrm{mdi}_k(T) \geq \mathrm{mdi}_k(T \setminus x) + 2 \geq f_k(n-1) + 2 > f_k(n).$$ 

If $T$ does not have a $k$-special pair, then $S_{p,\lfloor k/2 \rfloor+1} \nsubseteq T$.  By Lemmas~\ref{lem9}(i) and~\ref{lem6}(i), $T$ has a $k$-twin leaf $\ell'$ such that $\mathrm{mdi}_k(T) > \mathrm{mdi}_k(T \setminus \ell')$. Clearly, $S_{p,\lfloor k/2 \rfloor+1} \nsubseteq T \setminus \ell'$ and thus $T \setminus \ell' \notin \mathrm{Add}_k^*(S_{p,\lfloor k/2 \rfloor+1},r-1)$. Therefore, the inductive hypothesis implies
$$\mathrm{mdi}_k(T) \geq \mathrm{mdi}_k(T \setminus \ell') + 1 > f_k(n-1) + 1 = f_k(n).$$ 
\end{proof}

\begin{lemma}\label{lem16}
Let $T$ be an $n$-vertex tree of diameter $k + 2$. The following hold:

(i) If $n = p \cdot (\lfloor k/2 \rfloor + 1) + 1$, then $\mathrm{mdi}_k(T) > f_k(n)$.

(ii)  If $n = p \cdot (\lfloor k/2 \rfloor + 1) + 1 + r$, where $r \in \big[\lfloor k/2 \rfloor\big]$, then $\mathrm{mdi}_k(T) \geq f_k(n)$, with equality if and only if $T \in \mathrm{Add}_k(\mathcal{B}_{p,\lfloor k/2 \rfloor+1},r-1)$.
\end{lemma}
\begin{proof}
\textbf{Statement (i).} We apply induction on $p$. The base case $p = 2$ holds trivially, as there are no $(k + 2)$-vertex trees of diameter $k + 2$. For $p \geq 3$, observe that $T \notin \mathcal{B}_{p,\lfloor k/2 \rfloor + 1}$. By Lemma~\ref{lem9}, every subtree of $T$ that has diameter $k + 2$ and contains no $k$-twin leaves has at most $|V(T)| - \lfloor k/2 \rfloor$ vertices. Thus, by Lemma~\ref{lem8}, $T$ is $\lfloor k/2 \rfloor$-reducible to some tree $T^*$. If $T^* \in \mathcal{B}_{p-1,\lfloor k/2 \rfloor+1}$, then Lemmas~\ref{lem6}(i) and~\ref{lem1}(i) imply $$\mathrm{mdi}_k(T) \geq \mathrm{mdi}_k(T^*) + \lfloor k/2 \rfloor = f_k(n - \lfloor k/2 \rfloor) + \lfloor k/2 \rfloor  > f_k(n).$$ Otherwise, by Lemma~\ref{lem9}, there exists a tree $T^{**}$ such that $T^* \succ_k T^{**}$. The inductive hypothesis and Lemma~\ref{lem15} imply $\mathrm{mdi}_k(T^{**}) \geq f_k(n - \lfloor k/2 \rfloor - 1)$. By Lemma~\ref{lem1}(ii), $$\mathrm{mdi}_k(T) \geq \mathrm{mdi}_k(T^{**}) + \lfloor k/2 \rfloor + 1 \geq f_k(n - \lfloor k/2 \rfloor - 1) + \lfloor k/2 \rfloor + 1 > f_k(n).$$

\textbf{Statement (ii).} For a fixed $p \geq 2$, we apply induction on $r$. The base case $r = 1$ holds because, by Lemma~\ref{lem9}(ii), any tree $T' \notin \mathrm{Add}_k(\mathcal{B}_{p,\lfloor k/2 \rfloor+1},0)$ is reducible to a tree $T''$ with $p \cdot (\lfloor k/2 \rfloor+1) + 1$ vertices. Statement (i) then implies $\mathrm{mdi}_k(T'') > f_k(n-1)$, and thus $\mathrm{mdi}_k(T') > f_k(n)$. Now assume $r \geq 2$.

\textbf{Case 1.} $T \notin \mathrm{Add}_k(\mathcal{B}_{p,\lfloor k/2 \rfloor+1},r-1)$. If $T$ has a diametral $k$-twin $\ell$, then by Lemma~\ref{lem6}(ii) and the inductive hypothesis,
$$\mathrm{mdi}_k(T) \geq \mathrm{mdi}_k(T \setminus \ell) + 2 > f_k(n-1) + 2 > f_k(n).$$

If $T$ has no diametral $k$-twins, then it must contain a non-central branching vertex or a pendant path with at most $\lfloor k/2 \rfloor$ vertices connected to a central vertex. In either case, $T$ has a $k$-twin leaf. Let $\ell'$ be such a leaf with minimum eccentricity; one can see that $T \setminus \ell' \notin \mathrm{Add}_k(\mathcal{B}_{p,\lfloor k/2 \rfloor+1},r-2)$. By the inductive hypothesis, $$\mathrm{mdi}_k(T) \geq \mathrm{mdi}_k(T \setminus \ell') + 1 > f_k(n-1) + 1 \geq f_k(n).$$

\textbf{Case 2.} $T \in \mathrm{Add}_k(\mathcal{B}_{p,\lfloor k/2 \rfloor+1},r-1)$. By Lemmas~\ref{lem6}(i) and~\ref{lem9}, there exists a $k$-twin leaf $\ell'' \in V(T)$ such that $\mathrm{mdi}_k(T \setminus \ell'') \leq \mathrm{mdi}_k(T) - 1$. Since $\ell''$ is not a diametral leaf, there is a unique $k$-MDIS containing it; thus, $\mathrm{mdi}_k(T \setminus \ell'') = \mathrm{mdi}_k(T) - 1$ by Lemma~\ref{lem6}(i). Furthermore, $T \setminus \ell'' \in \mathrm{Add}_k(\mathcal{B}_{p,\lfloor k/2 \rfloor+1},r-2)$; hence, the statement holds.
\end{proof}

\begin{lemma}\label{lem17}
Let $T$ be an $n$-vertex tree with $\mathrm{diam}(T) \geq k + 3$. Then $\mathrm{mdi}_k(T) > f_k(n)$.
\end{lemma}

\begin{proof}
We apply induction on~$n$, using the same approach as in Lemma~\ref{lem14}. Consider a branching vertex $w \in V(T)$ adjacent to pendant paths $P_a$ and $P_b$ (with $a \leq b$) having endvertices $a_1$ and $b_1$, with respective neighbors $a_2$ and $b_2$. We may assume that, by the choice of $w$, the parameter $a$ is minimal; thus, there exists a diametral path not containing $a_1$ and $\mathrm{diam}(T \setminus a_1) = \mathrm{diam}(T)$. We now consider several cases:

\textbf{Case~1. $a + b \leq k$.} There exists a vertex $x$ on $P_b$ such that $N_k[x] = N_k[a_1]$. The inductive hypothesis and Lemma~\ref{lem6}(i) imply $$\mathrm{mdi}_k(T) \geq \mathrm{mdi}_k(T \setminus a_1) + 1 > f_k(n-1) + 1 \geq f_k(n).$$

\textbf{Case~2. $a = b = \lfloor k/2 \rfloor + 1$.} Consider the tree $T'$ obtained from $T$ by deleting $P_b$. Note that $\mathrm{diam}(T') = \mathrm{diam}(T)$. The inductive hypothesis and Lemma~\ref{lem7} imply
$$\mathrm{mdi}_k(T) \geq \mathrm{mdi}_k(T') + \lfloor k/2 \rfloor > f_k(n-\lfloor k/2 \rfloor-1) + \lfloor k/2 \rfloor = f_k(n).$$

\textbf{Case~3.} $b \geq \lfloor k/2 \rfloor + 2$ and $\mathrm{diam}(T) \geq k + 4$. Following the argument in Lemma~\ref{lem14} (Case 3), one can check that $\mathrm{mdi}_k^*(T,b_1) \geq 1$. Since $\mathrm{diam}(T \setminus b_1) \geq k + 3$, the inductive hypothesis implies $$\mathrm{mdi}_k(T) \geq \mathrm{mdi}_k(T \setminus b_1) + 1 > f_k(n-1) + 1 \geq f_k(n).$$

\textbf{Case~4.} $b \geq \lfloor k/2 \rfloor + 2$ and $\mathrm{diam}(T) = k + 3$. If $T$ has at least two branching vertices, we may assume that it contains two distinct pendant paths with at least $\lfloor k/2 \rfloor + 2$ vertices each, implying $\mathrm{diam}(T) \geq k + 4$, a contradiction. Therefore, $T$ has a unique branching vertex $w$, and all its pendant paths, except possibly $P_a$, contain $\lfloor k/2 \rfloor+2$ vertices. If $\mathrm{mdi}_k^*(T,a_1) \geq 1$, then since $\mathrm{diam}(T \setminus a_1) = \mathrm{diam}(T)$, the inductive hypothesis implies $$\mathrm{mdi}_k(T) \geq \mathrm{mdi}_k(T \setminus a_1) + 1 > f_k(n-1) + 1 \geq f_k(n).$$ If $\mathrm{mdi}_k^*(T,a_1) = 0$, then every $k$-MDIS containing $a_1$ must also contain $b_1$. This is possible only if $\mathrm{dist}(a_1,b_1) = k + 1$ and $a = \lfloor k/2 \rfloor$. Following the argument in Lemma~\ref{lem14} (Case 3), one can check that $\mathrm{mdi}_k^*(T,b_1) \geq 1$ and thus $\mathrm{mdi}_k(T) > \mathrm{mdi}_k(T \setminus b_1)$. If $\deg(w) \geq 4$, then $\mathrm{diam}(T \setminus b_1) = \mathrm{diam}(T)$; hence, $\mathrm{mdi}_k(T) > f_k(n)$ by the inductive hypothesis. If $\deg(w) = 3$, then $\mathrm{diam}(T \setminus b_1) = k + 2$ and $\lfloor k/2 \rfloor + 1$ divides $|V(T \setminus b_1)| - 1$. Therefore, Lemmas~\ref{lem4} and~\ref{lem16}(i) imply $$\mathrm{mdi}_k(T) \geq \mathrm{mdi}_k(T \setminus b_1) + 1 > f_k(n-1) + 1 > f_k(n).$$ 
\end{proof}

Lemmas~\ref{lem14}--\ref{lem17} imply the following fact.

\begin{theorem}\label{thm3}
For all odd $k \geq 3$ and all $n \geq 1$, every $k$-minimal $n$-vertex tree has exactly $f_k(n)$ $k$-MDISs. Furthermore, the following hold:

(i) If $1 \leq n \leq k + 1$, all $n$-vertex trees are $k$-minimal;

(ii) If $n \geq k + 2$ and $\lfloor k/2 \rfloor + 1$ divides $n - 1$, then $S_{p,\lfloor k/2 \rfloor+1}$ is the only $k$-minimal $n$-vertex tree, where $p = \lfloor (n-1)/(\lfloor k/2 \rfloor+1) \rfloor$;

(iii) If $n \geq k + 2$ and $\lfloor k/2 \rfloor + 1$ does not divide $n - 1$, then  $$\mathrm{Add}_k^*(S_{p,\lfloor k/2 \rfloor+1}, r) \cup \mathrm{Add}_k(\mathcal{B}_{p,\lfloor k/2 \rfloor+1},r-1)$$
is the class of all  $k$-minimal $n$-vertex trees, where $p = \lfloor (n-1)/(\lfloor k/2 \rfloor+1) \rfloor$ and 
$r = (n - 1) \bmod (\lfloor k/2 \rfloor + 1)$.
\end{theorem}

Remarkably, Theorem~3 also holds for $k = 1$, because $\lfloor k/2 \rfloor + 1 = 1$ divides $n - 1 \geq 2$, and thus the only 1-minimal $n$-vertex tree is isomorphic to the star graph $S_{n-1,1}$. 

For $p \geq 3$, let $S'''_{p,2}$ be a tree obtained from $S_{p-1,2}$ by adding a leaf adjacent to a vertex of degree 2. The next observation follows immediately from Theorem~3 and from the fact that for all $p,h \geq 1$ we have $|\mathcal{B}_{p,h
}| = \lfloor p/2 \rfloor$.

\begin{corollary}
The following hold:

(i) If $1 \leq n \leq 4$, all $n$-vertex trees are $3$-minimal.

(ii) If $n \geq 5$ is odd, the only $3$-minimal $n$-vertex tree is $S_{(n-1)/2,2}$.

(iii) If $n \geq 6$ is even, then  $\{S'_{n/2,2},S'''_{n/2,2}\} \cup \mathcal{B}_{(n-2)/2,2}$ is the class of $3$-minimal $n$-vertex trees. Moreover, there are $\lfloor (n+6)/4 \rfloor$ such trees.
\end{corollary} 

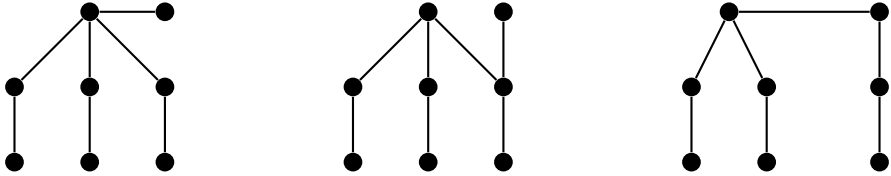
\begin{figure}[h!]
\center
\begin{tikzpicture}
  [scale=1,auto=left,every node/.style={circle}]

\node[fill = black,inner sep=2.5pt] (a0) at (1,1) {};
\node[fill = black,inner sep=2.5pt] (a1) at (0,0) {};
\node[fill = black,inner sep=2.5pt] (a2) at (0,-1) {};
\node[fill = black,inner sep=2.5pt] (a3) at (1,0) {};
\node[fill = black,inner sep=2.5pt] (a4) at (1,-1) {};
\node[fill = black,inner sep=2.5pt] (a5) at (2,0) {};
\node[fill = black,inner sep=2.5pt] (a6) at (2,-1) {};
\node[fill = black,inner sep=2.5pt] (a7) at (2,1) {};

\node[fill = black,inner sep=2.5pt] (b0) at (5.5,1) {};
\node[fill = black,inner sep=2.5pt] (b1) at (4.5,0) {};
\node[fill = black,inner sep=2.5pt] (b2) at (4.5,-1) {};
\node[fill = black,inner sep=2.5pt] (b3) at (5.5,0) {};
\node[fill = black,inner sep=2.5pt] (b4) at (5.5,-1) {};
\node[fill = black,inner sep=2.5pt] (b5) at (6.5,0) {};
\node[fill = black,inner sep=2.5pt] (b6) at (6.5,-1) {};
\node[fill = black,inner sep=2.5pt] (b7) at (6.5,1) {};

\node[fill = black,inner sep=2.5pt] (c0) at (9,-1) {};
\node[fill = black,inner sep=2.5pt] (c1) at (9,0) {};
\node[fill = black,inner sep=2.5pt] (c2) at (10,-1) {};
\node[fill = black,inner sep=2.5pt] (c3) at (10,0) {};
\node[fill = black,inner sep=2.5pt] (c4) at (9.5,1) {};
\node[fill = black,inner sep=2.5pt] (c5) at (11.5,1) {};
\node[fill = black,inner sep=2.5pt] (c6) at (11.5,0) {};
\node[fill = black,inner sep=2.5pt] (c7) at (11.5,-1) {};

\foreach \from/\to in {
a0/a1,a1/a2,a0/a3,a3/a4,a0/a5,a5/a6, a0/a7, b0/b1,b1/b2,b0/b3,b3/b4,b0/b5,b5/b6,b5/b7,c0/c1,c2/c3,c1/c4,c3/c4,c4/c5,c5/c6,c6/c7}
    \draw[thick] (\from) -- (\to);

\end{tikzpicture}
\caption{\small All non-isomorphic $3$-minimal 8-vertex trees.}
\end{figure}

Finally, we show that if a $k$-minimal tree is not unique, then there are linearly many such trees.

\begin{theorem}\label{thm4}
For all odd $k \geq 3$ and all $n \geq 4$, if $(k+1)/2$ does not divide $n - 1$, then the number of non-isomorphic $n$-vertex $k$-minimal trees is between $n / (k + 1)$ and $t(k^2) \cdot n$.
\end{theorem}

\begin{proof}
If $n \leq k + 1$, the statement holds trivially; let $n = p \cdot (\lfloor k/2 \rfloor + 1) + 1 + r$ for some $p \geq 2$ and $r \in \big[ \lfloor k/2 \rfloor \big]$. The lower bound follows immediately from the fact that $|\mathrm{Add}_k(S_{p,\lfloor k/2 \rfloor+1},1)| \geq 2$, and for all $r \geq 1$ we have $$|\mathrm{Add}_k(\mathcal{B}_{p,\lfloor k/2 \rfloor+1},r - 1)| \geq |\mathrm{Add}_k(\mathcal{B}_{p,\lfloor k/2 \rfloor+1},0)| = |\mathcal{B}_{p,\lfloor k/2 \rfloor+1}| = \bigg\lfloor \frac p2 \bigg\rfloor \geq \frac{n}{k+1} - 1.$$

We now prove the upper bound. Using the argument from Theorem~2, one can verify that $|\mathrm{Add}_k^*(S_{p,\lfloor k/2 \rfloor+1}, r)| < t(k^2)$. Let $p = p_1 + p_2$ with $\min(p_1,p_2) \geq 1$, and let $$\mathfrak{B}_{p_1,r} = \mathrm{Add}_k(B_{p_1,p_2,\lfloor k/2 \rfloor+1},r - 1).$$ Since $|\mathcal{B}_{p,\lfloor k/2 \rfloor+1}| < n/2$, it suffices to show that $|\mathfrak{B}_{p_1,r}| < t(k^2)$ for every $r \in \big[\lfloor k/2 \rfloor\big]$ and $1 \leq p_1 \leq \lfloor p/2 \rfloor$. We may assume that $p \geq k + 1$; otherwise, all trees from $\mathfrak{B}_{p_1,r}$ have fewer than $k^2$ vertices and the inequality holds trivially. We now consider three cases.

\textbf{Case 1.} If $p_2 > p_1 \geq \lfloor k/2 \rfloor + 1$, consider a function $$F_1: \mathfrak{B}_{p_1,r} \longrightarrow \mathrm{Add}_k(B_{\lfloor k/2 \rfloor,\lfloor k/2 \rfloor+1,\lfloor k/2 \rfloor+1},r - 1)$$

that maps a tree $T \in \mathfrak{B}_{p_1,r}$ to a tree $F_1(T)$ such that its central vertices are connected to  $p_1 - \lfloor k/2 \rfloor$ and $p_2 - \lfloor k/2 \rfloor - 1$ fewer pendant $(\lfloor k/2 \rfloor + 1)$-paths, respectively.

\textbf{Case 2.} If $p_2 > \lfloor k/2 \rfloor \geq p_1$, consider a function $$F_2: \mathfrak{B}_{p_1,r} \longrightarrow \mathrm{Add}_k(B_{p_1,\lfloor k/2 \rfloor+1,\lfloor k/2 \rfloor+1},r - 1)$$ 
that maps a tree $T \in \mathfrak{B}_{p_1,r}$ to a tree $F_2(T)$ such that one of its central vertices is connected to $p_2 - \lfloor k/2 \rfloor - 1$ fewer pendant $(\lfloor k/2 \rfloor + 1)$-paths.

\textbf{Case 3.} If $p_1 = p_2$, consider a function $$F_3: \mathfrak{B}_{p_1,r} \longrightarrow \mathrm{Add}_k(B_{\lfloor k/2 \rfloor,\lfloor k/2 \rfloor,\lfloor k/2 \rfloor+1},r - 1)$$
that maps a tree $T \in \mathfrak{B}_{p_1,r}$ to a tree $F_3(T)$ such that both its central vertices are connected to $p_1 - \lfloor k/2 \rfloor$ fewer pendant $(\lfloor k/2 \rfloor + 1)$-paths.

One can check that for all $i \in \{1,2,3\}$, $F_i$ is an injection and the trees $F_i(T)$ have fewer than $k^2$ vertices; this implies the upper bound.
\end{proof}
 
\section{Discussion}

It follows from Theorems~\ref{thm1} and~\ref{thm3} that every isolate-free forest on at least 5 vertices and with the minimum possible number of $k$-MDISs is a $k$-minimal tree:

\begin{corollary}
For all $k \geq 1$ and $n \geq 5$, every disconnected isolate-free $n$-vertex forest has more than $f_k(n)$ $k$-MDISs. The forest $2P_2$ has $f_2(4) + 1$ $2$-MDISs and $f_k(4)$ $k$-MDISs for all $k \geq 3$.
\end{corollary}

The sharp lower bound for the number of $k$-MDISs in the class of all isolate-free graphs seems to be more difficult to obtain. This class requires new approaches, because Lemma~\ref{lem3} (which was our most powerful tool) does not apply: if one removes a vertex belonging to a cycle, the graph remains connected, but the distances between the pairs of remaining vertices may change. For example, the simple $8$-cycle has four $3$-MDISs (because every vertex of it belongs to a unique $3$-MDIS of cardinality 2), but its $7$-vertex subgraph obtained by removing a vertex has seven $3$-MDISs by Lemma~\ref{lem5}(iii).

For all $k \geq 2$, one can easily find examples of $n$-vertex bipartite graphs with $\lceil n/2 \rceil$ $k$-MDISs (such as the simple $(2k+2)$-cycle or the $(k+1)$-dimensional hypercube). We believe that the bound $\lceil n/2 \rceil$ is the best possible for all $k \geq 2$. It is known~\cite{CW24} that the sharp lower bound for the number of 1-MDISs in the class of bipartite twin-free graphs is $\lceil n/2 \rceil + 1$. 

\begin{question}
Does there exist for some $n,k \geq 2$ an $n$-vertex bipartite isolate-free graph with fewer than $\lceil n/2 \rceil$ $k$-MDISs?
\end{question}

The class of circulant graphs provides some examples of $n$-vertex isolate-free graphs with fewer than $n/2$ $k$-MDISs. For example, the 24-vertex 7-regular circulant graph $C_{24}^{2,3,9,12}$ has exactly eight $2$-MDISs, each containing three vertices. Therefore, the next questions arise:

\begin{question}
For a given $k \geq 2$, what is the asymptotic behavior of the minimum possible number of $k$-MDISs in the class of $n$-vertex isolate-free graphs?
\end{question}

\begin{question}
Is it true that for every $k \geq 2$ and sufficiently large $n$, the following holds: if an isolate-free (bipartite) graph has the minimum possible number of $k$-MDISs among all $n$-vertex isolate-free (bipartite) graphs, then each of its vertices belongs to exactly one $k$-MDIS?
\end{question}

\section*{Acknowledgments}

The article was prepared within the framework of the Basic Research Program at the
National Research University Higher School of Economics (HSE).


\begin{thebibliography}{9}	

\bibitem{MM60}
R.E. Miller, D.E. Muller, A problem of maximum consistent subsets, IBM Research Report RC-240, J.T. Watson Research Center, New York, USA, 1960.

\bibitem{MM65}
J.W. Moon, L. Moser, On cliques in graphs, Isr. J. Math. 3 (1965) 23--28.

\bibitem{F87} Z. Füredi, The number of maximal independent sets in connected graphs, J. Graph Theory 11 (1987) 463--470.

\bibitem{GGG88} J.R. Griggs, C.M. Grinstead, D.R. Guichard, The number of maximal independent sets in a connected graph, Discrete Math. 68 (1988) 211--220.

\bibitem{HT93} M. Hujter, Z. Tuza, The number of maximal independent sets in triangle-free graphs, SIAM J. Discrete Math. 6 (1993) 284--288.

\bibitem{KGD08} K.M. Koh, C.Y. Goh, F.M. Dong, The maximum number of maximal independent sets in unicyclic connected graphs, Discrete Math. 308 (2008)
3761--3769.

\bibitem{L94} J. Liu, Constraints on the number of maximal independent sets in graphs, J. Graph Theory 18 (1994) 195--204.

\bibitem{W86}
H.S. Wilf, The number of maximal independent sets in a tree, SIAM J. Algebr. Discrete Methods 7 (1986) 125--130.

\bibitem{S88} B.E. Sagan, A note on independent sets in trees, SIAM J. Discrete Math. 1 (1988) 105--108.

\bibitem{CJ97} G.J. Chang and M.J. Jou. The number of maximal independent sets in connected triangle-free graphs. Discrete Math. 197 (1999) 169--178.

\bibitem{TM18} D.S.Taletskii and D.S.Malyshev. Trees without twin-leaves with the smallest number of
maximal independent sets. Diskret. Mat. 30:4 (2018) 115--134.

\bibitem{CW24} S.Cambie, S.Wagner. The Minimum Number of Maximal Independent Sets in Twin-Free Graphs. The Electronic Journal of Combinatorics. 31 (2024). P4.71.

\bibitem{T23}  D.S.Taletskii On Trees with a Given Diameter and the Extremal Number of Distance-$k$ Independent Sets.  Journal of Applied and Industrial Mathematics. 17:3 (2023) 664--677.

\bibitem{FPSV13} A.Frendrup, A.S.Pedersen, A.A.Sapozhenko, P.D.Vestergaard, “Merrifield–Simmons
index and minimum number of independent sets in short trees”, Ars Combin., 111 (2013).
85--95.

\bibitem{T21} D.Taletskii. Trees of Diameter 6 and 7 with Minimum Number of Independent Sets. Math. Notes. 109 (2021) 280--291.

\bibitem{E13} R.Euler, P.Oleksik, Z.Skupień, Counting Maximal Distance-Independent Sets in Grid Graphs. Discussiones Mathematicae Graph Theory. 33:3 (2013) 531--557.

\bibitem{BBMB24} D.Bouhata, S.Bouam, H.Moumen, B.Benreguia,  C.Arar. Self-stabilizing algorithms for computing maximal distance-2 independent sets and minimal dominating sets in networks. Ingénierie des Systèmes d'Information. 29(2) (2024) 581--590.

\bibitem{EGM14} H.Eto, F.Guo, E.Miyano. Distance-$d$ independent set problems for bipartite and chordal graphs. Journal of Combinatorial Optimization. 27:1 (2014) 88--99.
\end{thebibliography}
\end{document}